\documentclass[12pt]{amsart}
\usepackage{color}
\usepackage{bbm}

\usepackage{amssymb}
\usepackage{graphicx}  
\DeclareGraphicsExtensions{.pstex,.eps}

\newcommand{\R}{{\mathbb{R}}}
\newcommand{\M}{{\mathcal{M}}}

\newcommand{\F}{{\mathcal F}} 
\newcommand{\CC}{{\mathbb C}}

\newcommand{\Q}{{\mathcal Q}}
\newcommand{\RR}{{\mathbb R}}

\newcommand{\NN}{{\mathbb N}}

\newcommand{\p}{\partial} 
\newcommand{\supp}{\operatorname{supp}}

\renewcommand{\Im}{\mathop{\rm Im}\nolimits}

\theoremstyle{plain}

\newtheorem{thm}{Theorem}
\newtheorem{prop}{Proposition}[section]

\newtheorem{lemma}[prop]{Lemma}

\theoremstyle{definition}

\numberwithin{equation}{section}

\def\squarebox#1{\hbox to #1{\hfill\vbox to #1{\vfill}}}

\newcommand{\la}{\langle}
\newcommand{\ra}{\rangle}

\usepackage{amsxtra}

\title[Quasilinear Schr\"odinger equations]
{Quasilinear Schr\"odinger equations II: small data and cubic nonlinearities}

\author[J.L. Marzuola]
{Jeremy L. Marzuola}

\author[J. Metcalfe]
{Jason Metcalfe}

\author[D. Tataru]
{Daniel Tataru}

\address{Department of Mathematics, University of North Carolina-Chapel Hill \\
Phillips Hall, Chapel Hill, NC  27599-3250, USA}
\email{marzuola@email.unc.edu}

\address{Department of Mathematics, University of North Carolina-Chapel Hill \\
Phillips Hall, Chapel Hill, NC  27599-3250, USA}
\email{metcalfe@email.unc.edu}

\address{Mathematics Department, University of California \\
Evans Hall, Berkeley, CA 94720, USA}
\email{tataru@math.berkeley.edu}

\begin{document}

\begin{abstract}

  In part I of this project we examined low regularity local
  well-posedness for generic quasilinear Schr\"odinger equations with
  small data.  This improved, in the small data regime, the preceding
  results of Kenig, Ponce, and Vega as well as Kenig, Ponce, Rolvung,
  and Vega.  In the setting of quadratic interactions, the
  (translation invariant) function spaces which were utilized
  incorporated an $l^1$ summability over cubes in order to account for
  Mizohata's integrability condition, which is a necessary condition
  for the $L^2$ well-posedness for the linearized equation.  For cubic
  interactions, this integrability condition meshes better with the
  inherent $L^2$ nature of the Schr\"odinger equation, and such
  summability is not required.  Thus we are able to prove small data 
well-posedness in $H^s$ spaces.

\end{abstract}

\maketitle 

\section{Introduction}
We shall examine local well-posedness for quasilinear Schr\"odinger
equations with cubic interactions and a Cauchy datum in a low
regularity Sobolev space.  In particular, we examine
\begin{equation}
\label{eqn:quasi}
\left\{ \begin{array}{l}
i \partial_t u + g^{jk} (u,\nabla u ) \p_j \p_ku = 
F(u,\nabla u) , \quad u:
\RR \times \RR^d \to \CC^m \\ \\
u(0,x) = u_0 (x)
\end{array} \right.
\end{equation}
where
\[
g : \CC^m \times (\CC^m)^d \to
\RR^{d \times d}, \qquad 
F: \CC^m \times (\CC^m)^d \to \CC^m
\]
are smooth functions which satisfy
\[
g(y,z) = I_d+O(|y|^2+|z|^2), \qquad 
F(y,z)= O(|y|^3+|z|^3) \text{ near } (y,z) = (0,0).  
\]
We also examine
\begin{equation}
\label{eqn:quasi1}
\left\{ \begin{array}{l}
i \partial_t u + \p_j g^{jk} (u)  \p_ku = 
F(u,\nabla u) , \ u:
\RR \times \RR^d \to \CC^m \\ \\
u(0,x) = u_0 (x),
\end{array} \right. 
\end{equation}
where $g$ depends only on $u$.  The latter class of equations can be
obtained by differentiating the first equation.  Indeed, for $u$
solving \eqref{eqn:quasi}, the vector $(u,\nabla u)$ solves an
equation of the latter type.  The latter is written in divergence
form, which follows easily for this class of equations as the terms
obtained when commuting $g$ with the derivative can be absorbed into
$F(u,\nabla u)$.

 The case  of  small initial data and quadratic, rather than
cubic, nonlinear terms, was considered in  \cite{MMT3}. There,
 local well-posedness was established for data
of Sobolev-type regularity $s>\frac{d}{2}+3$ for \eqref{eqn:quasi} and
$s>\frac{d}{2}+2$ for \eqref{eqn:quasi1}.  This represented a
significant improvement in regularity over the previous results
\cite{KPV}, \cite{KPRV1, KPRV2}, though data of arbitrary size was
examined there.  A more complete history of such problems can be found
in \cite{MMT3}, \cite{KPV}, and \cite{LinPon}.

In the quadratic case above, it is insufficient, however, to simply work in Sobolev
spaces.  Indeed, one encounters Mizohata's integrability condition
\cite{Miz1, Miz2, Miz3}, \cite{Ich}, \cite{MMT1} which says that for
the linear equation
\[
(i\partial_t + \Delta_g)v = A_i(x)\partial_i v,
\]
a necessary condition for $L^2$ well-posedness is an integrability of
the real part of $A$ along the Hamiltonian flow of the leading order
operator.  The potentials encountered when linearizing, e.g.,
\eqref{eqn:quasi} with quadratic interactions do not necessarily
satisfy such a condition even with arbitrarily high regularities.  For
this reason, the initial data spaces need to incorporate some decay.
The translation invariant approach of \cite{MMT3} was to require a
summability over cubes, which was inspired by the earlier work
\cite{BT} on semilinear derivative Schr\"odinger equations.

In the case of only cubic and higher order interactions, the scenario
is much simpler, which is what we shall demonstrate here.  When
linearizing \eqref{eqn:quasi} as stated, the potential that is
encountered is $O(|u|^2+|\nabla u|^2)$, for which integrability
follows easily from energy estimates.  As such, the additional $l^1$
summability which was previously required is no longer needed, and the
initial datum can be simply taken to be a member of a Sobolev space.
This notion was previously explored in \cite{KPV2, KPV3}, and we seek
to improve on that by considering rough initial data.

Our main result is precisely that \eqref{eqn:quasi} and
\eqref{eqn:quasi1} are locally well-posed
for $u_0(x)\in H^s$.

\begin{thm}
\label{thm:main1}
a) Let $s > \frac{d+5}2$. Then there exists   $\epsilon > 0$
sufficiently small  such that for all initial data $u_0$ with
\begin{eqnarray*}
\| u_0 \|_{H^s} \leq \epsilon ,
\end{eqnarray*}
the equation \eqref{eqn:quasi} is locally well-posed in $H^s
(\RR^d)$ on the time interval $I = [0,1]$.

b) The same result holds for the equation \eqref{eqn:quasi1}
with $s > \frac{d+3}2 $.
\end{thm}

In the theorem, well-posedness is taken to include the existence
of a local solution, uniqueness, and continuous dependence on the
initial datum.

The above theorem also holds for ultrahyperbolic operators, as in
\cite{KPRV1, KPRV2}, where $g(0)$ is of a different signature.  Our
method of proof of the local smoothing estimates uses a wedge
decomposition, and the estimates are proved in each coordinate
direction separately.  Thus, trivial adjustments of the sign of the
multiplier in the corresponding directions of opposite signature will
permit said results.

This article is organized as follows.  In Section 2, we set our
notations and describe the function spaces in which we shall work.
Section 3 is devoted to proving the required nonlinear estimates.
Section 4 contains the proof of our main linear estimate, which is a
variant of the classical local smoothing estimate for the
Schr\"odinger equation.  In Section 5  we prove Theorem
\ref{thm:main1}.

{\sc Acknowledgments.} {
The second author was supported in part by NSF grant DMS-1054289.  
The third author was supported in part by  NSF grant DMS-0801261
as well as by the Miller Foundation.}


\section{Function Spaces and Notations}
\label{sec:boot}

The function spaces and notations presented in this section are the
same as those established in \cite{MMT3}.  We repeat them for easy
reference.

For a function $u(t,x)$ or $u(x)$, we let $\hat{u}=\mathcal{F}u$ denote the
Fourier transform in the spatial variables $x$.  A function $u$ is
said to be localized at frequency $2^i$ if $\hat{u}(t,\xi)$ is
supported in $\R\times [2^{i-1},2^{i+1}]$.  We shall use a
Littlewood-Paley decomposition of the spatial frequencies,
\[\sum_{i=0}^\infty S_i(D)=1,\]
where $S_i$ localizes to frequency $2^i$ for $i>0$ and to frequencies
$|\xi|\le 2$ for $i=0$.  Note that we are working in
nonhomogeneous spaces.  We set
\[u_i = S_i u, \quad S_{\le N}u=\sum_{i=0}^N S_i u,\quad S_{\ge N}u = \sum_{i=N}^\infty
S_i u.\]

For each $j\in \NN$, we let $\Q_j$ denote a partition of $\R^d$ into
cubes of side length $2^j$, and we let $\{\chi_Q\}$ denote an
associated partition of unity.  For a translation invariant
Sobolev-type space $U$, we set $l^p_j U$ to be the Banach space with
associated norm
\[\|u\|_{l^p_j U}^p = \sum_{Q\in \Q_j} \|\chi_Q u\|_U^p\]
with the obvious modification for $p=\infty$.

Within such norms, the smooth partition of unity with compactly
supported functions $\chi_Q$ can be replaced by cutoffs which are
frequency localized as the associated Schwartz tails decay rapidly
away from $Q$.  Thus, the components can be taken to retain the
frequency localization of the function $u$.

Motived by the well-known local smoothing estimates, we define
\[
  \| u \|_{X}  =   \sup_{l} \sup_{Q \in \mathcal{Q}_l} 
2^{-\frac{l}{2}} \| u \|_{L^2L^2 ([0,1] \times Q)}.
\]
Here and throughout, $L^p L^q$ represents $L^p_t L^q_x$.

To measure the forcing terms, we use an atomic space $Y$ satisfying $X=Y^*$, see
\cite{MMT3}.  
A function $a$ is an atom in $Y$ if there is
a $j\ge 0$ and a $Q\in \Q_j$ so that $a$ is supported in $[0,1]\times
Q$ and
\[
\|a\|_{L^2([0,1]\times Q)}\lesssim 2^{-\frac{j}{2}}.
\]
Then we define $Y$ as linear combinations of the form
\[f=\sum_k c_k a_k,\quad \sum |c_k|<\infty,\quad a_k \text{ atoms}\]
with norm
\[
\| f \|_{Y} = \inf \{\sum |c_k|\,:\, f=\sum_k c_k a_k,\, a_k \text{ atoms} \}.
\]

For solutions which are localized to frequency $2^j$, we shall
encorporate the typical half derivative smoothing by working within
\[
X_j = 2^{-\frac{j}{2}}X\cap L^\infty L^2
\]
with  norm
\[
\|u\|_{X_j} = 2^{\frac{j}{2}}\|u\|_X + \|u\|_{L^\infty L^2}.
\]
We incorporate the Sobolev regularity and the cube summability by defining
\[
\|u\|^2_{l^p X^s} = \sum_{j\ge 0} 2^{2js} \|S_j u\|^2_{l^p_j X_j}.
\]
For the purposes of this paper, we will be working primarily in
$l^2 X^s$, as opposed to our previous paper \cite{MMT3} where we
had to work in $l^1 X^s$ due to incompatibilities with the Mizohata
condition.  On occasion we will also use the slightly larger space
\[
\|u\|^2_{X^s} = \sum_j 2^{2js} \|S_j u\|^2_{X_j}.
\]

We analogously define
\[
Y_j = 2^{\frac{j}{2}}Y+L^1L^2
\]
which has norm
\[\|f\|_{Y_j} = \inf_{f=2^{\frac{j}{2}}f_1+f_2} \|f_1\|_Y +
\|f_2\|_{L^1L^2}\]
and
\[\|f\|^2_{l^pY^s} = \sum_j 2^{2js} \|S_j f\|^2_{l^p_jY_j}.\]
Here, we shall be working primarily within $l^2Y^s$, though we note
that the Minkowski integral inequality gives 
\[\|f\|_{l^2Y^s} \lesssim \|f\|_{Y^s}\]
where
\[
\|f\|_{Y^s}^2 = \sum_j 2^{2js} \|S_j f\|^2_{Y_j}.
\]
Hence, the cube summation will largely be ignored for the forcing terms.

We also note that for any $j$, we have 
\[
 \sup_{Q \in \mathcal{Q}_j} 
2^{-\frac{j}{2}} \| u \|_{L^2L^2 ([0,1] \times Q)} \leq \| u \|_X,
\] 
hence
\begin{eqnarray}
\label{eqn:Ybound}
\| u \|_Y \lesssim 2^{\frac{j}{2}} \| u \|_{l^1_j L^2 L^2 } .
\end{eqnarray}
This bound will come in handy at several places later on.

In our multilinear estimates, we shall use frequency envelopes.  Consider a
Sobolev-type space $U$ for which we have
\[
\|u\|_U^2 \sim \sum_{k=0}^\infty \|S_k u\|_U^2.
\]
A frequency envelope for a function $u \in U$ is a positive $l^2$ sequence, $\{a_j\}$,
with 
\[
\|S_j u\|_U \le a_j \|u\|_U.
\]
We shall only permit slowly varying frequency envelopes.  Thus, we
require $a_0\approx 1$ and 
\begin{equation}\label{slowlyvarying}
a_j \le 2^{\delta |j-k|} a_k,\quad j,k\ge 0,\quad 0<\delta\ll 1.
\end{equation}
The constant $\delta$ shall be chosen later and only depends on $s$
and the dimension $d$.  Such frequency envelopes always exist.  E.g.,
one may choose
\begin{equation}\label{env}a_j = 2^{-\delta j} + \|u\|_U^{-1}\sup_k 2^{-\delta|j-k|} \|S_k
u\|_U.\end{equation}


\section{Multilinear and nonlinear estimates}
\label{sec:mult}

This section contains our main multilinear estimates.  The first
proposition is essentially from \cite{MMT1}, though it is somewhat
simpler as the summability over cubes is easier.

\begin{prop}
\label{fe:nonlinear}
Let $\sigma > \frac{d}2$, and $u,v \in l^2 X^\sigma$ with admissible
frequency envelopes
$a_k$, respectively $b_k$. Then the $l^2 X^\sigma$ spaces satisfy the
algebra type property
  \begin{equation}\label{fe:u_squared}
\| S_k (u v)\|_{l^2 X^\sigma} \lesssim (a_k+b_k)\|u\|_{l^2 X^\sigma} \|v\|_{l^2 X^\sigma}, 
  \end{equation}
as well as the Moser type estimate
  \begin{equation}\label{fe:moser}
\|S_k F(u) \|_{l^2 X^\sigma} \lesssim a_k \|u\|_{l^2
  X^\sigma}(1+\|u\|_{l^2 X^\sigma}) c(\|u\|_{L^\infty}). 
  \end{equation}
for all smooth $F$ with $F(0)= 0$.
\end{prop}

\begin{proof}
To prove \eqref{fe:u_squared}, we consider $\|S_k(S_i u S_j
v)\|_{l^2_k X_k}$.  We must have
either $i\ge k-4$ or $j\ge k-4$; by symmetry we shall assume the former.  Since $i\ge
k-4$, we have $X_i\subset X_k$.  And we naively use the fact that there are
$\approx 2^{d(i-k)}$ cubes of sidelength $2^k$ contained in one of
sidelength $2^i$.   Thus, by Bernstein's inequality, we
have
\begin{align}
\label{eqn:hhl}
\|S_k(S_iu S_j v)\|_{l^2_k X_k} &\lesssim 2^{\frac{d}{2}(i-k)}\|S_i
u\|_{l^2_i X_i} \|S_j v\|_{L^\infty}\\
&\lesssim 2^{\frac{d}{2}(i-k)} 2^{j\frac{d}{2}} \|S_i u\|_{l^2_i X_i} \|S_j
v\|_{l^2_j L^\infty L^2}. \notag
\end{align} 
It follows that 
\[\|S_k(S_i u S_j v)\|_{l^2 X^\sigma} \lesssim a_i b_j
2^{i(\frac{d}{2}-\sigma)}2^{j(\frac{d}{2}-\sigma)} 2^{k(\sigma-\frac{d}{2})}
\|u\|_{l^2 X^\sigma}\|v\|_{l^2 X^\sigma}.\]
As $\sigma>\frac{d}{2}$, the estimate \eqref{fe:u_squared} now follows after summation in $i, j$.
When $k-4\le i\le k+4$, we simply use that the Cauchy-Schwarz inequality
implies that $\sum_j 2^{j(\frac{d}{2}-\sigma)}b_j$ is $O(1)$.  When
$i>k+4$, we also use that
$\sum_{i>k} 2^{i(\frac{d}{2}-\sigma)} a_i \lesssim
2^{k(\frac{d}{2}-\sigma)}a_k$
provided the $\delta$ in \eqref{slowlyvarying} satisfies $0<\delta<\sigma-\frac{d}{2}$.


We now examine \eqref{fe:moser} and replace the discrete
Littlewood-Paley decomposition by a continuous one
\[Id = S_0 + \int_0^\infty S_k \,dk.\]
Abusing notation, we will set $u_0=S_0 u$ for now.  Then,
using the Fundamental Theorem of Calculus, we can write
\begin{equation}
  \label{ftc}
  S_kF(u) = S_kF(u_0) + \int_0^\infty S_k(u_{k_1}F'(u_{<k_1}))\,dk_1.
\end{equation}

For the first term of \eqref{ftc}, we note that
\[\|\partial^\alpha F(u_0)\|_{l^2_0 X_0}\lesssim \|u_0\|_{l^2_0 X_0} c(\|u_0\|_{L^\infty}).\]
Thus,
\[\|S_k F(u_0)\|_{l^2_k X_k}\lesssim 2^{k/2} \|S_kF(u_0)\|_{l^2_0 X_0} \lesssim
2^{-Nk} \|u_0\|_{l^2_0 X_0}c(\|u_0\|_{L^\infty})\]
for any $N$, from which the $l^2 X^\sigma$ bound follows easily.

We split the integral in the second term of \eqref{ftc} into two
regions.  For $k_1\ge k-4$, we use that $X_{k_1}\subset X_k$ and have
\begin{multline*}2^{k\sigma}\Bigl\|\int_{k-4}^\infty
S_k(u_{k_1}F'(u_{<k_1}))\,dk_1\Bigr\|_{l^2_k X_k} \\\lesssim \int_{k-4}^\infty
2^{-(k_1-k)(\sigma-\frac{d}{2})} \|u_{k_1}\|_{l^2 X^\sigma}
\|F'(u_{<k_1})\|_{L^\infty}\,dk_1\end{multline*}
by using \eqref{eqn:hhl}.  
From this, the desired bound follows if $\delta$ in
\eqref{slowlyvarying} is chosen small enough.  This case is significantly
simpler here than its analogous case in \cite{MMT3}  because we
are now dealing with $l^2$ sums rather than $l^1$.

For $k_1<k-4$, we note that
\[S_k(u_{k_1}F'(u_{<k_1}))=S_k(u_{k_1}\tilde{S}_kF'(u_{<k_1}))\]
where $\tilde{S}_k$ also localizes to frequency $2^k$ and $S_k
\tilde{S_k}=S_k$.  The chain rule allows us to estimate
\[\|\tilde{S}_k F'(u_{<k_1})\|_{L^\infty} \lesssim 2^{-N(k-k_1)}
c(\|u_{<k_1}\|_{L^\infty}).\]
Thus,
\begin{align*}\|S_k(u_{k_1}F'(u_{<k_1}))\|_{l^2_k X_k} &\lesssim 2^{\frac{k-k_1}{2}}
\|u_{k_1}\|_{l^2_{k_1} X_{k_1}} \|\tilde{S}_k F'(u_{<k_1})\|_{L^\infty}\\
&\lesssim 2^{-N(k-k_1)} \|u_{k_1}\|_{l^2_{k_1} X_{k_1}} c(\|u_{<k_1}\|_{L^\infty}).
\end{align*}
In turn, this leads to
\[\|S_k(u_{k_1}F'(u_{<k_1}))\|_{l^2 X^\sigma} \lesssim 2^{-N(k-k_1)} a_{k_1}
\|u\|_{l^2 X^\sigma} c(\|u\|_{L^\infty}),\]
where as $N$ was arbitrary, we allowed its value to change from line
to line,
and after integration, \eqref{fe:moser} is proved.
\end{proof}






Replacing the bilinear estimates which were used in \cite{MMT3} are
the following trilinear estimates.  In particular, we note that here
we only require $s>\frac{d+3}{2}$, which accounts for the improvement
in regularity as compared to the quadratic interactions explored in \cite{MMT3}.

\begin{prop}
Let $s>\frac{d+3}{2}$, and suppose that $u\in l^2X^{\sigma-1}$, $v, w\in
l^2 X^{s-1}$ with frequency envelopes $\{a_k\}$, $\{b_k\}$, and
$\{c_k\}$ respectively.  Then for $0\le \sigma\le s$, we have
\begin{equation}\label{tri}
\| S_k(u v w)\|_{l^2 Y^\sigma} \lesssim (a_k+b_k+c_k) \|
u\|_{X^{\sigma-1}} \|v\|_{l^2 X^{s-1}} \|w\|_{l^2
  X^{s-1}}.
\end{equation}
If $0\le \sigma\le s-1$ and if $u\in l^2 X^\sigma$, $v\in l^2X^{s-2}$
and $w\in l^2X^{s-1}$, then
\begin{equation}\label{tri2}
\| S_k(u v w)\|_{l^2 Y^\sigma} \lesssim (a_k+b_k+c_k) \|
u\|_{l^2 X^{\sigma}} \|v\|_{l^2 X^{s-2}} \|w\|_{l^2
  X^{s-1}}.
\end{equation}
Finally, for $0\le \sigma\le s$ and $u\in l^2X^{\sigma-2}$ and $v,w\in l^2
X^s$, and
\begin{equation}\label{tri3}
\| S_k(u S_{\ge k-4}(v w))\|_{l^2 Y^\sigma} \lesssim (a_k+b_k+c_k) \|
u\|_{l^2 X^{\sigma-2}} \|v\|_{l^2 X^{s}} \|w\|_{l^2
  X^{s}}.
\end{equation}
\end{prop}

\begin{proof}
 For the proof  we need to establish trilinear  estimates of the form 
\begin{equation}
\label{tri1}\|S_k(u_i v_j w_l)\|_{l^2_k Y_k} \lesssim c_{kijl}   \|u_i\|_{X_i}
\|v_j\|_{l^2_j X_j} \|w_l\|_{l^2_l X_l}.
\end{equation}
for several different cases of frequency balances.

{\bf Case A.  $|i-k|\le 4$ and $j,l\le k+4$. }   We
shall assume, without loss, that $j\ge l$.  Here, we measure the
output in $Y$, use $X$ for the highest frequency factor, and $L^\infty
L^2$ for the lower frequency factors.  Using \eqref{eqn:Ybound} we
have 
\[
LHS\eqref{tri1}
\lesssim 2^{-k/2} \|S_k(u_i v_j w_l)\|_Y \lesssim 2^{-k/2}2^{j/2}
\|S_k(u_i v_j w_l)\|_{l^1_j L^2L^2}.
\] 
By Bernstein's inequality and using that $l<j$, this yields 
\begin{align*}LHS \eqref{tri1} \ &\lesssim 2^{-k/2} 2^{j/2}  2^{jd/2}
\|v_j\|_{l^2_j L^\infty L^2} \|u_i w_l\|_{l^2_l L^2 L^2}\\
&\lesssim 2^{-k/2} 2^{j\frac{d+1}{2}} 2^{l\frac{d+1}{2}} \|u_i\|_{X}
\|v_j\|_{l^2_l L^\infty L^2} \|w_l\|_{l^2_l L^\infty L^2},
\end{align*}
which in turn gives
\begin{equation}\label{tri1a}
 c_{kijl} \lesssim 2^{-k/2} 2^{-i/2} 2^{l\frac{d+1}{2}} 2^{j\frac{d+1}{2}}.
\end{equation}

To prove \eqref{tri} for this type of interactions
we use \eqref{tri1a} and the frequency envelopes of $u$, $v$, and $w$
to bound 
\begin{multline*}
 \|S_k(u_i v_j w_l)\|_{l^2 Y^\sigma} 
\\\lesssim  a_i 2^{j\Bigl(\frac{d+3}{2}-s\Bigr)}  b_j
2^{l\Bigl(\frac{d+3}{2}-s\Bigr)}  
c_l    \|u\|_{X^{\sigma-1}} \|v\|_{l^2 X^{s-1}} \|w\|_{l^2 X^{s-1}},
\end{multline*}
which is easily summed over $i,j,l$ in the appropriate ranges. 
The case where the same frequency balance holds but with
the roles of $i$, $j$ and $l$ interchanged is similar but simpler.

 Now consider \eqref{tri2}. The worst case is $i \geq j \geq l$ if 
$\sigma \leq s-2$, respectively  $j \geq i \geq l$ if 
$s-2 < \sigma \leq s-1$. In both cases \eqref{tri2} is weaker than \eqref{tri}.

Finally, for the estimate \eqref{tri3} we must have either $j > k-8$ or $l > k-8$,
in which case \eqref{tri3} is also weaker than \eqref{tri}.

\medskip
{\bf Case B. $i,j > k-4$, $|i-j|<4$, $l<i+8$.}
Grouping terms in the following manner
\[\Bigl|\int \overline{z_k} u_i v_j w_l \, dx\,dt\Bigr| 
\lesssim  \|v_j\|_{l^2_j X_j} \|z_k u_i w_l\|_{l^2_j Y_j}\]
and
in view of the duality relation $l^2_j X_j = l^2_j Y_j^*$, from the 
bound \eqref{tri1a} we directly obtain
\begin{equation}\label{tri1b}
c_{ijkl}  \lesssim 2^{-\frac{i}2} 2^{-\frac{j}2} 2^{\frac{d+1}{2}k} 2^{\frac{d+1}{2}l}.
\end{equation}
Then the  estimates \eqref{tri}, \eqref{tri2} and \eqref{tri3}  follow again  by passing
to frequency envelopes and summing over the appropriate ranges of
$i,j,l$.  Here, again, we have handled explicitly the worst case, and
those cases where the roles of $i,j,l$ are interchanged are simpler.

We remark that in the case of \eqref{tri3} we need to deal with the additional
multiplier $S_{>k-4}$. However, all our function spaces are translation invariant,
and the kernel of $S_{>k-4}$ is a bounded measure. Thus it can easily be disposed of.
\end{proof}

We shall also utilize the following bound on commutators when, e.g.,
we decompose into functions whose Fourier transforms are supported in
wedges about the coordinate axes.  The proof closely resembles that of
\cite{MMT3}, though care is taken to permit $s>\frac{d+3}{2}$ rather
than $s>\frac{d}{2}+2$.
\begin{prop}\label{propComm} We assume
\[g^{kl}-\delta^{kl} = h^{kl}(w(t,x))\]
where $h(z)=O(|z|^2)$ near $|z|=0$.
For $s > \frac{d+3}2$ and any multiplier $B\in S^0$ we have
\begin{equation}\label{fe:xxycom}
\| \nabla [S_{< k-4} g,B(D)] \nabla S_k u  \|_{Y_k} \lesssim 
 \|w\|^2_{l^2 X^{s}} (1+\|w\|_{l^2 X^s}) c(\|w\|_{L^\infty}) \|S_k u\|_{X_k}.
\end{equation}
\end{prop}

\begin{proof}
In \cite[Proposition 3.2]{MMT3}, it was shown that
\[\nabla[S_{<k-4} g, B(D)]\nabla S_k u = L(\nabla S_{<k-4} g, \nabla
S_k u),\]
where $L$ is a translation invariant operator satisfying
\[L(f,g)(x)=\int f(x+y)g(x+z)w(y,z)\,dy\,dz,\quad w\in L^1.\]
We shall not reproduce this proof here.

Given this representation, as we are working in translation invariant
spaces, the bound \eqref{fe:xxycom} follows immediately from
\eqref{fe:moser} and \eqref{tri}.
\end{proof}

%


\section{Local Energy Decay}
\label{sec:LED}
Here we consider the main linear estimate which we shall employ in
order to prove Theorem~\ref{thm:main1}.  This is a variant on the
well-known local smoothing estimates for the linear Schr\"odinger
equation.  The metric is assumed to be
a composition
\begin{equation}\label{composition}g^{kl}-\delta^{kl} = h^{kl}(w(t,x))\end{equation}
where $h(z)=O(|z|^2)$ near $|z|=0$.  The focus, then, is on a
frequency localized linear Schr\"odinger equation
\begin{equation} \label{freq-loc}
(i \partial_t + \partial_k g^{kl}_{< j-4}  \partial_l ) u_j = f_j, \qquad 
u_j(0)= u_{0j} .
\end{equation}
In the sequel, we shall pass to this setting.

The main estimate is:
\begin{prop}
\label{prop:morawetz}
Assume that the coefficients $g^{kl}$ in  \eqref{freq-loc} 
satisfy \eqref{composition} with
\begin{equation}\label{small_hyp}
\|w\|_{l^2 X^s} \ll 1
\end{equation}
for some $s>\frac{d+3}{2}$.  
Let   $u_j$  be a solution to  \eqref{freq-loc}  which
is localized at frequency $2^j$. Then the following estimate 
holds:
\begin{equation} \label{l1le}
 \| u_{j} \|_{l^2_j X_j} \lesssim \| u_{0j} \|_{L^2} +  
 \| f_{j} \|_{l^2_j Y_j} .
\end{equation}
\end{prop}

The proof is closely related to that given in \cite{MMT3} and uses a
positive commutator argument.  The multipliers will be first order
differential operators with smooth coefficients which are localized at
frequency $\lesssim 1$; precisely, for estimating waves at frequency $2^j$
we will use multipliers $\M$ that are differential operators having the form
\begin{equation}\label{defM}
i 2^j \M = a^k(x) \partial_k + \partial_k a^k(x)
\end{equation}
with uniform bounds on $a$ and its derivatives.

A key spot where the proof here differs from that of \cite{MMT3} is in
the following lemma which is used to dismiss the $g-I$ contribution to
the commutator.

\begin{lemma}\label{commLemma}
  Let $A = \partial_k S_{<j-4}(h^{kl}(w(t,x))) \partial_l $ with $h$
  as above, and let $\M$ be as above.  Suppose that $s>\frac{d+3}{2}$.
  Then we have
\begin{equation}
| \la [A,\M]u_j,u_j\ra| \lesssim  \|w\|^2_{l^2 X^s} (1+\|w\|_{l^2 X^s})c(\|w\|_{L^\infty}) \|u_j\|_{X_j}^2.
\end{equation}
\end{lemma}

\begin{proof}
We note that we can abstractly write $[A,\M]$ as
\[
i[A,\M] = 2^{-j} ( \nabla ( h \nabla a + a \nabla h) \nabla + \nabla h \nabla^2 a +
h \nabla^3 a )
\]
where $h$ abbreviates $S_{<j-4}(h(w(t,x)))$.  As the $a$ factors are
bounded and low frequency, the worst term to bound is
\[
2^{-j} \la a \nabla h \nabla u_j, \nabla u_j \ra.
\]
By duality, we only require
\[\|(\nabla h) \nabla u_j\|_{Y_j}\lesssim \|w\|_{l^2 X^s}^2 (1+\|w\|_{l^2 X^s})c(\|w\|_{L^\infty})
\|u_j\|_{X_j},\]
which is a conseqeunce of \eqref{fe:moser} and \eqref{tri}.
\end{proof}

We now continue the proof of the proposition.
For $Q\in \Q_l$, $0\le l\le j$ fixed, we begin by proving that a
solution $u_j$ to \eqref{freq-loc} satisfies
\begin{equation}\label{le-l}
  2^{j-l}\| u_j\|_{L^2(Q)}^2 \lesssim  \ 
\|u_j\|_{L^\infty L^2}^2  +
  \|u_j\|_{X_j} \|f_j\|_{Y_j}  
  + (2^{-j} +  \| w\|^2_{l^2 X^s}) \|u_j\|_{X_j}^2.
\end{equation}

We start by abstracting and studying solutions to 
\[(D_t+A)u=f,\quad u(0)=u_0\]
where $A$ is a self-adjoint operator and $D_t=\frac{1}{i}\partial_t$.  First of all, we have an
energy-type estimate:
\begin{equation}
\| u\|_{L^\infty_t L^2_x}^2 \lesssim \| u_0\|_{L^2}^2
+ \| u\|_{X_j} \|f\|_{Y_j}. 
\label{eest}\end{equation}
See \cite[Lemma 4.2]{MMT3}.  For a self-adjoint multiplier $\M$, we
also have
\begin{equation}\label{abstractcomm}
\frac{d}{dt}\la u,\M u\ra = -2\Im\la(D_t+A)u,\M u\ra + \la
i[A,\M]u,u\ra.
\end{equation}
Thus, for $u=u_j$ and $A=-\partial_k g^{kl}_{<j-4} \partial_l$, 
if we can provide a multiplier $\M$ so that
\begin{enumerate}
\item $\|\M u\|_{L^2_x}\lesssim \|u\|_{L^2_x}$,
\item $\|\M u\|_X \lesssim \|u\|_X$,
\item $i\la [A,\M]u,u\ra \gtrsim 2^{j-\ell} \|u\|_{L^2L^2([0,1]\times Q)}^2
- O(2^{-j} +\|w\|^2_{l^2 X^s}) \| u\|_{X_j}^2$,
\end{enumerate}
\eqref{le-l} would follow.

We first do this when the Fourier transform of the solution $u_j$ is
restricted to a small angle
\begin{equation}
\supp \hat u_j \subset \{ |\xi| \lesssim \xi_1 \}.
\label{uj-ang}\end{equation}
Here, we have done so about the first coordinate axis, though the
modifications to obtain the same about the other axes are trivial.
We take, without loss of generality due to translation invariance,
$Q=\{|x_j|\le 2^l\,:\, j=1,\dots,d\}$, and we set $m$ to be a smooth,
bounded, increasing function such that $m'(s)=\psi^2(s)$ where $\psi$ is a
Schwartz function localized at frequencies $\lesssim 1$, and $\psi
\sim 1$ for $|s|\le 1$.  We rescale $m$ and set $m_l(s)=m(2^{-l}s)$.
Then, we fix
\[\M = \frac{1}{i2^j}\Bigl(m_l(x_1)\partial_1 + \partial_1 m_l(x_1)\Bigr).\]

Properties (1) and (2) are immediate due to the frequency
localizations of $u=u_j$ and $m_l$ as well as the boundedness of
$m_l$.  By Lemma \ref{commLemma}, it suffices to consider property (3)
for $A=-\Delta$.  This yields
\[i2^j [-\Delta,\M]= -2^{-l+2} \partial_1 \psi^2(2^{-l}x_1)\partial_1
+ O(1),\]
and hence
\[i2^j\la [-\Delta,\M]u,u\ra = 2^{-l+2} \|\psi(2^{-l}x_1)\partial_1
u_j\|^2_{L^2L^2} + O(\|u_j\|^2_{L^2L^2}).\]
Utilizing our assumption \eqref{uj-ang}, it follows that
\[2^{-l} 2^j \|\psi(2^{-l} x_1)u_j\|_{L^2L^2}^2 \lesssim i\la
[-\Delta,\M]u_j,u_j\ra + 2^{-j} O(\|u_j\|^2_{L^2L^2}),\]
which yields (3) when combined with Lemma \ref{commLemma}.

We proceed to reducing to the the case that \eqref{uj-ang} holds.  We
let $\{\theta_k(\omega)\}_{k=1}^d$ be a partition of unity,
\[\sum_k \theta_k(\omega)=1,\quad \omega\in \mathbb{S}^{d-1},\]
where $\theta_k(\omega)$ is supported in a small angle about the $k$th
coordinate axis.  Then, we can set $u_{j,k}=\Theta_{j,k}u_j$ where
\[\F \Theta_{j,k}u =
\theta_k\Bigl(\frac{\xi}{|\xi|}\Bigr)\sum_{j-1\le l\le j+1}
\phi_l(\xi) \hat{u}(t,\xi).\]
Here $\phi_l(\xi)$ is the Fourier multiplier associated to $S_l(D)$.
We see that
\[(D_t+A)u_{j,k} = -\Theta_{j,k}f_j + [A,\Theta_{j,k}]u_j.\]

The operators $\Theta_{j,k}$ are bounded on $L^2$ by Plancherel's
theorem.  And as the kernel has Schwartz-type decay outside a ball of
radius $2^{-j}$, it is easy to see that they are also bounded on $X$,
and via duality, $Y$.  Hence, by applying $\M$, suitably adapted to
the correct coordinate axis, to $u_{j,k}$ and summing over $k$, we
obtain
\begin{multline}\label{le-l1}
2^{j-l}\| u_j\|_{L^2(Q)}^2 \lesssim 
\|u_j\|_{L^\infty L^2}^2 +
  \|u_j\|_{X_j} (  \|f_{j}\|_{Y_j}   + \!\! \sum_k \| [A,\Theta_{j,k}] u_j \|_{Y_j})
\\   +  (2^{-j} + \|w\|^2_{l^2 X^s}) \|u_j\|_{X_j}^2. 
\end{multline}
Upon estimating the commutator term with \eqref{fe:xxycom},
\eqref{le-l} follows.  If we then take the supremum over $Q\in \Q_l$
and subsequently over $l$, yields
\[\|u_j\|^2_{X_j} \lesssim \|u_j\|^2_{L^\infty L^2} + \|u_j\|_{X_j}
\|f_j\|_{Y_j} + (2^{-j} + \|w\|^2_{l^2 X^s})\|u_j\|_{X_j}^2,\]
provided $\|w\|^2_{l^2 X^s}$ is bounded.  Using \eqref{small_hyp}, we
may bootstrap to obtain
\begin{equation}\label{needsEest}\|u_j\|^2_{X_j} \lesssim
  \|u_j\|_{L^\infty L^2}^2 + \|u_j\|_{X_j} \|f_j\|_{Y_j}\end{equation}
for $j$ sufficiently large.  On the other hand, for small $j$, we have
the trivial estimate
\[\|u\|_{X_j} \lesssim \|u\|_{L^\infty L^2},\]
which yields the same.

By applying \eqref{eest} to \eqref{needsEest}, we have shown
\begin{equation}\label{noSum}\|u_j\|_{X_j} \lesssim \|u_{0j}\|_{L^2} + \|f_j\|_{Y_j}.\end{equation}
We now finish the proof by incorporating the summation over cubes.  We
let $\{\chi_Q\}$ denote a partition via functions which are localized to
frequencies $\lesssim 1$ which are associated to cubes $Q$ of scale
$M2^j$.  We also assume $|\nabla^k \chi_Q|\lesssim (2^j M)^{-k}$,
$k=1,2$.  Thus,
\[(D_t+A) (\chi_Q u_m) = -\chi_Q f_j + [A,\chi_Q]u_j.\]
Applying \eqref{noSum} to $\chi_Q u_m$, we obtain
\begin{equation}\label{firstSum}\sum_Q \|\chi_Q u_j\|^2_{X_j} \lesssim \|u_{0j}\|_2^2 + \sum_{Q}
\|\chi_Q f_j\|^2_{Y_j} + \sum_Q \|[A,\chi_Q] u_j\|^2_{L^1L^2}.\end{equation}
But
\[\sum_Q \|[A,\chi_Q]u_j\|^2_{L^1 L^2} \lesssim M^{-2} \sum_Q \|\chi_Q
u_j\|_{L^\infty L^2}^2 \lesssim M^{-2} \sum_Q \|\chi_Q u_j\|^2_{X_j}.\]
For $M$ sufficiently large, we can bootstrap the last term of
\eqref{firstSum}, and upon a straightforward transition to cubes of
scale $2^j$ rather than $M2^j$, \eqref{l1le} follows.


\section{Proof of 
Theorem \ref{thm:main1}}
\label{sec:proof}

We prove part (b) of the theorem, from which part (a) follows
immediately.

We first examine the linear equation
\begin{eqnarray}
\label{lin1}
\begin{cases}
  (i\partial_t + \partial_k g^{kl}\partial_l)u + V\nabla u + Wu=H,\\
u(0)=u_0
\end{cases}
\end{eqnarray}
under the assumption that
\[g^{kl}-\delta^{kl} = h^{kl}(u(t,x))\]
where $h(z)=O(|z|^2)$ near $|z|=0$.  We shall also assume that  there exist smooth 
functions $\tilde V$ and $\tilde W$ so that
\[
V =
\tilde{V}(\tilde{w}(t,x)),\qquad W=\tilde{W}(v(t,x),w(t,x), \tilde{w}(t,x))
\]
where the behavior of $\tilde V$ and $\tilde W$ near zero is given by
\[
\tilde{V}(z)=O(|z|^2), \qquad  \tilde{W}(\phi,\psi, \theta)
=\phi \cdot O(|\psi|)+O(|\theta|^2).
\]

Our main estimate is the
following:
\begin{prop}\label{p:lin1}
a) Assume that the metric $g$ and potentials $V$ and $W$ are as above, and they satisfy 
\[
\| w \|_{l^2 X^s} \ll 1, \quad \|\tilde{w}\|_{l^2 X^{s-1}} \ll 1 , \quad
\|v\|_{l^2 X^{s-2}} \ll 1  \qquad  s > \frac{d+3}2.
\]
Then the equation \eqref{lin1} is well-posed for initial data
$u_0 \in H^\sigma$ with $0 \leq \sigma \leq s-1$,
and we have the estimate
\begin{equation}\label{bd:lin1}
\| u\|_{l^2 X^\sigma} \lesssim \|u_0\|_{H^\sigma} + \|H\|_{l^2 Y^\sigma}.
\end{equation}

b) Assume in addition that $W=0$. Then the equation \eqref{lin1} is well-posed
 for initial data $u_0 \in H^\sigma$ with $0 \leq \sigma \leq s$, 
and the estimate \eqref{bd:lin1} holds.
\end{prop}

\begin{proof}
The existence part follows in a standard manner 
by approximating the metric, potentials and data by Schwartz
functions and passing to the limit on a subsequence,
provided that we have the uniform bound \eqref{bd:lin1}.
Hence the remainder of the proof is devoted to 
establishing \eqref{bd:lin1}.

For $u$ solving \eqref{lin1}, we see that $u_j$ solves
\[
\begin{cases}
 (i\partial_t + \partial_k g^{kl}_{<j-4}\partial_l)u_j = G_j + H_j,\\
u_j(0)=u_{0j},
\end{cases}
\]
where
\[G_j = -S_j\partial_k g^{kl}_{>j-4}\partial_l u - [S_j, \partial_k
g^{kl}_{<j-4}\partial_l]u - S_j V \nabla u - S_j Wu.\]
If we apply Proposition \ref{prop:morawetz} to each of these
equations, we see that
\[\|u\|^2_{l^2 X^\sigma} \lesssim \|u_0\|^2_{H^\sigma} +
\|H\|_{l^2Y^\sigma}^2 + \sum_j \|G_j\|^2_{l^2Y^\sigma}.\]

We claim that
\[\sum_j \|G_j\|^2_{l^2 Y^\sigma} \lesssim \|u\|^2_{l^2
  X^\sigma}\Bigl(\|w\|_{l^2 X^s}^4 + \|\tilde{w}\|_{l^2 X^{s-1}}^4 +
\|v\|_{l^2 X^{s-2}}^4\Bigr).\]
Indeed, for the first term in $G_j$ we apply \eqref{fe:moser} and
\eqref{tri3}. 
The bound for the second term in $G_j$ follows from \eqref{fe:xxycom}
and \eqref{fe:moser}.  The bounds for the last two terms of $G_j$ use
\eqref{tri} and \eqref{tri2} respectively in conjunction with \eqref{fe:moser}.
The differences in the range of permissible $\sigma$ in \eqref{tri}
versus \eqref{tri2} accounts precisely for the difference in parts (a)
and (b) of this proposition.
\end{proof}

\subsection{The iteration}
We now set up an iteration to solve \eqref{eqn:quasi1}.  We set $u^{(0)}\equiv 0$ and recursively
define $u^{(n+1)}$ to be the solution to
\begin{equation}
\label{iterate}
\left\{ \begin{array}{l}
(i \partial_t + \p_j g^{jk} (u^{(n)})  \p_k) u^{(n+1)} = 
F(u^{(n)},\nabla u^{(n)}) , \\ \\
u^{(n+1)} (0,x) = u_0 (x).
\end{array} \right. 
\end{equation}

As $F(y,z)=O(|y|^3+|z|^3)$ near the origin, it follows from
\eqref{fe:moser} with $\sigma=s-1$ and \eqref{tri} with $\sigma=s$
that 
\[\|F(u^{(n)}, \nabla u^{(n)})\|_{l^2 Y^s} \lesssim \|u^{(n)}\|^3_{l^2
  X^s}\]
provided $s>\frac{d+3}{2}$ and $\|u^{(n)}\|_{l^2 X^s}=O(1)$.  Using
this in each application of Proposition \ref{p:lin1}, we see, via induction, that
\begin{equation}\label{unifbdd}\|u^{(n)}\|_{l^2 X^s} \lesssim \|u_0\|_{H^s}\end{equation}
for each $n$ provided that $\|u_0\|_{H^s}$ is sufficiently small.

We now seek to show that the iteration converges in $l^2 X^{s-1}$.  To
this end, we note that $v^{(n+1)} = u^{(n+1)}-u^{(n)}$ solves
\begin{equation}
\label{iterate-diff}
\left\{ \begin{array}{l}
(i \partial_t + \p_j g^{jk} (u^{(n)})  \p_k) v^{(n+1)} =  V_n \nabla  v^{(n)}+ W_n v^{(n)},
 \\ 
v^{(n+1)}(0,x) = 0,
\end{array} \right. 
\end{equation}
where 
\begin{eqnarray*}
V_n & = & V_n(u^{(n)},\nabla u^{(n)},u^{(n-1)},\nabla u^{(n-1)}), \\
W_n & = & h_1(u^{(n)}, \nabla u^{(n)}, u^{(n-1)}, \nabla u^{(n-1)})  + h_2 (u^{(n)},u^{(n-1)})  \nabla^2 u^{(n)}.
\end{eqnarray*}
Here $V_n(z)=O(|z|^2)$, $h_1(z)=O(|z|^2)$, and $h_2(z)=O(|z|)$ near
$|z|=0$.   By Proposition \ref{p:lin1}, we have
\[\|v^{(n+1)}\|_{l^2 X^{s-1}} \lesssim \|V_n \nabla v^{(n)}\|_{l^2
  Y^{s-1}} + \|W_n v^{(n)}\|_{l^2Y^{s-1}}.\]  
By \eqref{tri} and \eqref{tri2} with $\sigma=s-1$, we have
\begin{align*}\|V_n \nabla v^{(n)}\|_{l^2 Y^{s-1}} &\lesssim \Bigl(\|u^{(n)}\|_{l^2
  X^{s}} + \|u^{(n-1)}\|_{l^2 X^{s}}\Bigr)^2 \|v^{(n)}\|_{l^2
  X^{s-1}},\\
\|W_n v^{(n)}\|_{l^2 Y^{s-1}}&\lesssim \Bigl(\|u^{(n)}\|_{l^2 X^s} +
\|u^{(n-1)}\|_{l^2 X^{s}}\Bigr)^2 \|v^{(n)}\|_{l^2 X^{s-1}}.
\end{align*}
Thus, it follows from \eqref{unifbdd} that
\begin{equation}
  \label{Cauchy}
  \|v^{(n+1)}\|_{l^2 X^{s-1}} \ll \|v^{(n)}\|_{l^2 X^{s-1}},
\end{equation}
which implies that the iteration converges in $l^2 X^{s-1}$ to a
function $u$ satisfying
\begin{equation}\label{goalBd}\|u\|_{l^2 X^s} \lesssim \|u_0\|_{H^s}.\end{equation}
This establishes the existence of a solution.

We next consider the question of uniqueness.  For two solutions $u^{(1)}$ and
$u^{(2)}$ to \eqref{eqn:quasi1}, we consider $v=u^{(2)}-u^{(1)}$.  Here, $v$ solves
\begin{equation}
\label{diff}
\left\{ \begin{array}{l}
(i \partial_t + \p_j g^{jk} (u^{(2)})  \p_k) v =  V \nabla  v+ W v,
 \\ 
v(0,x) = u^{(2)}-u^{(1)},
\end{array} \right. 
\end{equation}
where 
\begin{eqnarray*}
V & = & V(u^{(1)},\nabla u^{(1)},u^{(2)},\nabla u^{(2)}), \\
W_n & = & h_1(u^{(2)}, \nabla u^{(2)}, u^{(1)}, \nabla u^{(1)})  + h_2 (u^{(2)},u^{(1)})  \nabla^2 u^{(1)}.
\end{eqnarray*}
By Proposition \ref{p:lin1}, we have
\begin{equation}\label{weak-lip}
\|  u^{(2)} - u^{(1)}\|_{l^2 X^{s-1}} \lesssim \|  u^{(2)}(0) - u^{(1)}(0)\|_{H^{s-1}}
\end{equation}
from which uniqueness follows.

In order to show continuity of the map $u_0\to u$ from $H^s\to
l^2X^s$, we seek to strengthen the preceding argument.  To do so, we
shall use frequency envelopes.  Indeed, we first seek to strengthen
\eqref{goalBd} by showing that a frequency envelope for the data is
also a frequency envelope for the solution.

\begin{prop}\label{envelopes}  Let $s>\frac{d+3}{2}$, and
let $u \in l^2 X^s$ be a small data solution to \eqref{eqn:quasi1}, which satisfies
\eqref{goalBd}. If $\{a_j\}$ is an admissible frequency envelope for the initial
data $u_0$ in $H^s$, then $\{a_j\}$ is also a frequency envelope for $u$ 
in $l^2 X^s$.
\end{prop}

\begin{proof}
We set
\begin{equation}
  \label{bj}
  b_j = 2^{-\delta j} + \|u\|^{-1}_{l^2 X^s} \max_k 2^{-\delta |j-k|}
  \|S_k u\|_{l^2 X^s}
\end{equation}
and note that $\{b_j\}$ is a frequency envelope for $u$ in $l^2 X^s$.
We note that $u_j$ solves
\begin{equation}
  \label{paradiffEq}
 \begin{cases}
    (i\partial_t + \partial_k g^{kl}_{<j-4}\partial_l) u_j = S_j
    F(u,\nabla u) - S_j \partial_k g^{kl}_{>j-4}\partial_l u \\
 \quad \quad \quad \quad \quad \quad \quad \quad \quad \quad - [S_j, \partial_k g^{kl}_{<j-4}\partial_l ] u\\
u_j(0)  =(u_0)_j.
  \end{cases}
\end{equation}
Labelling the right side of the equation above $f_j$, we apply
Proposition \ref{p:lin1} and obtain
\begin{equation}\label{sjuBd}\|S_j u\|_{l^2 X^s} \lesssim a_j \|u_0\|_{H^s} + \|f_j\|_{l^2Y^s}.\end{equation}
We bound $f_j$ in a manner that is akin to the above.  This yields
\begin{equation}
  \label{fjbd}
  \|f_j\|_{l^2 Y^s} \lesssim b_j \|u\|^3_{l^2 X^s} c(\|u\|_{l^2 X^s}).
\end{equation}
Indeed,
provided $s>\frac{d+3}{2}$, we can apply \eqref{fe:moser} ($\sigma =
s-1$) and \eqref{tri} ($\sigma=s$) to bound the first term.  For the
second term we use \eqref{fe:moser} ($\sigma=s$) and \eqref{tri3}.
For the last term in $f_j$, we apply \eqref{fe:xxycom}. 

Applying \eqref{fjbd} in \eqref{sjuBd} we obtain
\[\|S_j u\|_{l^2 X^s} \lesssim a_j \|u_0\|_{H^s} + b_j
\|u\|_{l^2X^s}^3 c(\|u\|_{l^2 X^s}),\]
which yields
\[b_j \lesssim  a_j \|u_0\|_{H^s} \|u\|_{l^2 X^s}^{-1} + b_j
\|u\|_{l^2X^s}^2 c(\|u\|_{l^2 X^s}).\]
Using the smallness of $\|u\|_{l^2 X^s}$, we can bootstrap the second
term in the right.  Moreover, by the definition of $l^2X^s$, we have
$\|u_0\|_{H^s} \lesssim \|u\|_{l^2 X^s}$.  Thus, the above implies
$b_j\lesssim a_j$ as desired.
\end{proof}

We now proceed to show that the map $H^s\to l^2X^s$ given by $u_0
\mapsto u$ is continuous.  Let $\{u^{(n)}_0\}\subset H^s$ be a
sequence which converges to $u_0$ in $H^s$, and let $\{a^{(n)}_j\}$
and $\{a_j\}$ denote their respective frequency envelopes defined via
\eqref{env}.  It follows, thus, that $a^{(n)}_j \to a_j$ in $l^2$.
For any $\varepsilon>0$, there is a $N_\varepsilon$ so that
\[ \|a_{j}^{(n)}\|_{l^2(j>N_{\varepsilon})} \le \varepsilon,\quad
\|a_{j}\|_{l^2(j>N_{\varepsilon})}\le \varepsilon \]
uniformly in $n$.  The preceding proposition then yields
\begin{equation}\label{highFreq}
\begin{split}
\|u^{(n)}_{>N_\varepsilon}\|_{l^2 X^s}&\le \varepsilon
\|u^{(n)}\|_{l^2 X^s} \le C \varepsilon \|u^{(n)}_0\|_{H^s}
\\
\|u_{>N_\varepsilon}\|_{l^2 X^s}&\le \varepsilon
\|u\|_{l^2 X^s} \le C \varepsilon \|u_0\|_{H^s}
\end{split}
\end{equation}
where in the last step we applied \eqref{goalBd}.

To compare $u^{(n)}$, where $u^{(n)}$ is the solution to
\eqref{eqn:quasi1} with datum $u_0^{(n)}$, to $u$, we use \eqref{highFreq} for the high
frequencies and \eqref{weak-lip} for the low frequencies.  Indeed,
\begin{align*}
  \|u^{(n)}-u\|_{l^2X^s}&\lesssim
  \|S_{<N_\varepsilon}(u^{(n)}-u)\|_{l^2 X^s} +
  \|u^{(n)}_{>N_\varepsilon}\|_{l^2X^s} +
  \|u_{>N_{\varepsilon}}\|_{l^2X^s}\\
&\lesssim 2^{N_\varepsilon} \|S_{<N_\varepsilon} (u^{(n)}-u)\|_{l^2
  X^{s-1}} + \varepsilon \|u^{(n)}_0\|_{H^s} + \varepsilon
\|u_0\|_{H^s}\\
&\lesssim 2^{N_\varepsilon}
\|S_{<N_\varepsilon}(u_0^{(n)}-u_0)\|_{H^{s-1}} + \varepsilon
\|u^{(n)}_0\|_{H^s} + \varepsilon \|u_0\|_{H^s}.
\end{align*}
Letting $n\to \infty$ yields
\[\limsup_{n\to \infty} \|u^{(n)}-u\|_{l^2X^s} \lesssim \varepsilon \|u_0\|_{H^s},\]
and subsequently letting $\varepsilon\to 0$ yields the desired result.

We finish by showing the analog of \eqref{goalBd} for higher
frequencies:
\begin{equation}
  \label{goalBdHigh}
  \|u\|_{l^2 X^\sigma}\lesssim \|u_0\|_{H^\sigma},\quad \sigma\ge s
\end{equation}
assuming that $u_0\in H^\sigma$.  Differentiating the original
equation yields
\begin{multline*}
  (i\partial_t + \partial_j g^{jk}(u)\partial_k)(\partial_l u) =
  -(g^{jk})'(u)(\partial_j\partial_l u \partial_k u + \partial_l
  u \partial_j\partial_k u) \\- (g^{jk})''(u)(\partial_j u \partial_l
  u \partial_k u)
+ (\nabla_{z_1} F)(u,\nabla u)\cdot \nabla \partial_l u +
F_{z_0}(u,\nabla u) \partial_l u.
\end{multline*}
For $v_1=\nabla u$, we have
\[(i\partial_t +\partial_j g^{jk}(u)\partial_k)v_1 = G(u,\nabla
u)\nabla v_1
+ F_1(u,\nabla u)\]
where $G(z)=O(|z|^2)$ and $F_1(z)=O(|z|^3)$ near $z=0$.  By
\eqref{fe:moser} and \eqref{tri}, we have
\[\|G(u,\nabla u)\nabla v_1\|_{l^2 Y^s} \lesssim \|u\|^2_{l^2 X^s}
\|v_1\|_{l^2 X^s}, \quad \|F_1(u,\nabla u)\|_{l^2 Y^s}\lesssim \|u\|^3_{l^2X^s}.\]
And by Proposition \ref{p:lin1}, 
\[\|v_1\|_{l^2 X^s} \lesssim \|v_1(0)\|_{H^s} + \|u\|^3_{l^2 X^s},\]
and thus,
\[\|u\|_{l^2 X^{s+1}}\lesssim \|u(0)\|_{H^{s+1}} + \|u\|^3_{l^2
  X^s}.\]

Letting $v_n=\nabla^n u$, we see that $v_n$ solves
\[(i\partial_t + \partial_j g^{jk}(u)\partial_k)v_n = G(u,\nabla u)
\nabla v_n + F_n(u, \nabla u,\dots, \nabla^n u)\]
with $G$ as above.  Arguing inductively, we obtain
\[\|v_n\|_{l^2 X^s} \lesssim \|v_n(0)\|_{H^s} + \|u\|^3_{l^2
  X^{s+n-1}},\]
which gives
\[\|u\|_{l^2 X^{s+n}}\lesssim \|u(0)\|_{H^{s+n}} + \|u\|^3_{l^2 X^{s+n-1}}\]
from which the desired conclusion follows.

\end{document}